\documentclass[11pt,reqno]{amsart}

\usepackage[a4paper,margin=3.5cm]{geometry}
\usepackage[english]{babel}
\usepackage{amsmath,amsfonts,amssymb,amsthm,enumerate,url}
\usepackage{microtype}
\usepackage[all]{xy}

\usepackage{amssymb,bbm,mathrsfs}

\numberwithin{equation}{section}
\usepackage[colorlinks=true, pdfstartview=FitV, pagebackref=false]{hyperref}
\usepackage{xcolor}
\usepackage{tikz}
\usetikzlibrary{trees,arrows,positioning,automata,shadows,fit,shapes}
\usepackage{float}

\usepackage[numbers,sort]{natbib}

\newcommand{\e}{{\,{\rm e}}}
\newcommand{\var}{{\,{\rm Var}}}

\newcommand{\PP}{\ensuremath{\mathbb{P}}}

\newcommand{\EE}{\ensuremath{\mathbb{E}}}
\newcommand{\N}{\ensuremath{\mathbb{N}}}
\newcommand{\bZ}{\ensuremath{\mathbb{Z}}}
\newcommand{\cX}{\ensuremath{\mathcal{X}}}

\newcommand{\cD}{\ensuremath{\mathcal{D}}}
\newcommand{\cE}{\ensuremath{\mathcal{E}}}

\newcommand{\cG}{\ensuremath{\mathcal{G}}}
\newcommand{\cA}{\ensuremath{\mathcal{A}}}

\newcommand{\cB}{\ensuremath{\mathcal{B}}}
\newcommand{\bD}{\ensuremath{\mathbf{D}}}
\newcommand{\R}{\ensuremath{\mathbb{R}}}

\newcommand{\Z}{\ensuremath{Z}}
\newcommand{\X}{\ensuremath{X}}
\newcommand{\Y}{\ensuremath{Y}}
\newcommand{\W}{\ensuremath{\mathbf{W}}}
\def\ind{{\mathbbm{1}}}
\newcommand{\tmix}{\ensuremath{t_{\mathrm{mix}}}}

\newcommand{\given}{\, \big| \,}
\newcommand{\ti}{\widetilde}

\newcommand{\hx}{\ensuremath{\bar{x}}}

\newtheorem{lemma}{Lemma}
\newtheorem{theorem}{Theorem}
\newtheorem{proposition}[lemma]{Proposition}

\begin{document}

\title[]{Cutoff for a stratified random walk \\
on the hypercube}
\author[A. Ben-Hamou]{Anna Ben-Hamou}
 \address{A. Ben-Hamou\hfill\break
Sorbonne Universit\'e, LPSM\\
4, place Jussieu\\
75005 Paris, France.}
\email{anna.ben-hamou@upmc.fr}

\author[Y. Peres]{Yuval Peres}
\address{Y.\ Peres\hfill\break
Microsoft Research\\ One Microsoft Way\\ Redmond, WA 98052, USA.}
\email{peres@microsoft.com}

\keywords{Markov chains ; mixing times ; cutoff ; hypercube}
\subjclass[2010]{60J10}

\begin{abstract}
We consider the random walk on the hypercube which moves by picking an ordered pair $(i,j)$ of distinct coordinates uniformly at random and adding the bit at location $i$ to the bit at location $j$, modulo $2$. We show that this Markov chain has cutoff at time $\frac{3}{2}n\log n$ with window of size $n$, solving a question posed by Chung and Graham (1997). 
\end{abstract}

\maketitle

\section{Introduction}

\subsection*{Setting and main result}

Let $\cX=\{0,1\}^n\backslash \{\mathbf{0}\}$ and consider the Markov chain $\{\Z_t\}_{t\geq 0}$ on $\cX$ defined as follows: if the current state is $x$ and if $x(i)$ denotes the bit at the $i^{\text{th}}$ coordinate of $x$, then the walk proceeds by choosing uniformly at random an ordered pair $(i,j)$ of distinct coordinates, and replacing $x(j)$ by $x(j)+x(i)$ (mod~$2$). 

The transition matrix $P$ of this chain is symmetric, irreducible and aperiodic. Its stationary distribution $\pi$ is the uniform distribution over $\cX$, i.e. for all $x\in\cX$, $\pi(x)=\frac{1}{2^n -1}$. We are interested in the total-variation mixing time, defined as
\begin{eqnarray*}
\tmix(\varepsilon)&=& \min\left\{t\geq 0,\, d(t)\leq \varepsilon\right\}\, ,
\end{eqnarray*}
where $\displaystyle{d(t)=\max_{x\in\cX}d_x(t)}$ and $d_x(t)$ is the total-variation distance between $P^t(x,\cdot)$ and $\pi$:
\begin{eqnarray*}
d_x(t)&=&\sup_{A\subset\cX} \left(\pi(A)-P^t(x,A)\right)\;=\;\sum_{y\in\cX}\left(P^t(x,y)-\pi(y)\right)_+\, .
\end{eqnarray*}

\citet*{diaconis1996walks} showed that the log-Sobolev constant of $\{\Z_t\}_{t\geq 0}$ is $O(n^2)$, which yields an upper-bound of order $n^2\log n$ on the $\ell_2$-mixing time. They however conjectured that the right order for the total-variation mixing was $n\log n$. \citet*{chung1997stratified} confirmed this conjecture. They showed that the relaxation time of $\{\Z_t\}$ was of order $n$ (which yields a tight upper-bound of order $n^2$ for $\ell_2$-mixing) and that the total-variation mixing time $\tmix(\varepsilon)$ was smaller than $c_\varepsilon n\log n$ for some constant $c_\varepsilon$. They asked whether one could make this bound more precise and replace $c_\varepsilon$ by a universal constant which would not depend on $\varepsilon$. We answer this question positively by proving that the chain $\{\Z_t\}$ has cutoff at time $\frac{3}{2}n\log n$, with window of order $n$.

\begin{theorem}\label{thm:cutoff_p=2}
The chain $\{\Z_t\}$ has total-variation cutoff at time $\frac{3}{2}n\log n$ with window $n$,
\begin{eqnarray*}
\lim_{\alpha\to +\infty}\liminf_{n\to +\infty} d\left(\frac{3}{2}n\log n-\alpha n\right)&=&1\, ,
\end{eqnarray*}
and
\begin{eqnarray*}
\lim_{\alpha\to +\infty}\limsup_{n\to +\infty} d\left(\frac{3}{2}n\log n+\alpha n\right)&=&0\, .
\end{eqnarray*}
\end{theorem}

\subsection*{Motivation and related work}

The chain $\{\Z_t\}$ records the evolution of a single column in the following random walk on $\text{SL}_n(\bZ_2)$, the group of invertible $n\times n$ matrices with coefficients in $\bZ_2$: at each step, the walk moves by picking an ordered pair of distinct rows uniformly at random and adding the first one to the other, modulo $2$. This matrix random walk has received significant attention, both from group theoreticians and cryptologists. It was brought to our attention by Ron Rivest, who was mostly interested in computational mixing aspects, pertaining to \emph{authentication protocols}. In cryptography, an authentication protocol is a scheme involving two parties, a \emph{verifier} and a \emph{prover}, the goal of the verifier being to certify the identity of the prover (\emph{i.e.} to distinguish between an honest and a dishonest prover). A large family of authentication protocols, called \emph{time-based authentication protocols}, is based on the time needed by the prover to answer a challenge. The authentication is successful if and only if the correct answer is provided fast enough. The following protocol was proposed by \citet{sotiraki2016authentication}. Starting from the identity matrix in $\text{SL}_n(\bZ_2)$, the prover runs the above Markov chain driven by random row additions up to a certain time $t\in\N$. He makes the final matrix $A_t$ public (this is called a \emph{public key}), but only he knows the trajectory of the Markov chain. Then, whenever he wants to authenticate, the prover asks the verifier for a vector $x\in\{0,1\}^n$. The challenge is to quickly compute $y=A_t x$. As the prover knows the chain's trajectory, he can apply to $x$ the same row operations he has performed to create $A_t$ and provide the correct answer in time $t$. On the other hand, if $t$ is large enough, a dishonest party may not be able to distinguish, in polynomial time, $A_t$ from a uniformly randomly chosen matrix (we say that the chain is computationally mixed), and its best solution would be to perform usual matrix-vector multiplication, which typically takes about $n^2$ operations. Hence, if the prover chooses $t$ as the computational mixing time and if this time is shown to be much smaller than $n^2$, the verifier will be able to distinguish honest provers from dishonest parties.

Now, it is reasonable to assume that if each column is mixed by time $t$, then the matrix $A_t$ is computationally mixed and no dishonest party would be able to distinguish $A_t$ from uniform (in polynomial time) and thus to answer the challenge in less than $n^2$ units of time, hereby motivating the study of the chain $\{\Z_t\}$. 

The question of determining the total-variation mixing time of the matrix walk is still largely open. \citet*{diaconis1996walks} showed that the $\ell_2$-mixing time was $O(n^4)$, and the powerful results of \citet{kassabov2005kazhdan} yield the upper-bound $O(n^3)$, which is also the best known upper-bound in for total-variation mixing. By a simple counting argument, the total-variation mixing time can be lower bounded by $\Omega\left(\tfrac{n^2}{\log n}\right)$ (which is actually an estimate of the diameter of the underlying graph, see \citet{andren2007complexity,christofides2014asymptotic}). 

\subsection*{Outline of the paper}
Before proving Theorem \ref{thm:cutoff_p=2}, we first state some useful properties of the birth-and-death chain given by the Hamming weight of $\Z_t$. In particular, we show that this projected chain also has cutoff at $\frac{3}{2}n\log n$ (Section \ref{sec:hamming}). Section \ref{sec:proof} is then devoted to the proof of Theorem \ref{thm:cutoff_p=2}.

\section{The Hamming weight}
\label{sec:hamming}

For a vertex $x\in\cX$, we denote by $H(x)$ the Hamming weight of $x$, \textit{i.e.}
\begin{eqnarray*}
H(x)&=&\sum_{i=1}^n x(i)\, .
\end{eqnarray*}

Consider the birth-and-death chain $H_t:=H(\Z_t)$, and denote by $P_H$, $\pi_H$, and $d_H(\cdot)$ its transition matrix, stationary distribution, and total-variation distance to equilibrium. For $1\leq k\leq n$, we have
\begin{eqnarray*}
P_H(k,k+1)&=&\frac{k(n-k)}{n(n-1)}\, ,\\
P_H(k,k-1)&=&\frac{k(k-1)}{n(n-1)}\, ,\\
P_H(k,k)&=&\frac{n-k}{n}\, ,
\end{eqnarray*}
and
\begin{eqnarray*}
\pi_H(k) &=& \frac{\binom{n}{k}}{2^n-1}\, .
\end{eqnarray*}
The hitting time of state $k$ is defined as 
\begin{eqnarray*}
T_k&=& \min\left\{t\geq 0,\, H_t=k\right\}\, .
\end{eqnarray*}
One standard result in birth-and-death chains is that, for $2\leq \ell\leq n$,
\begin{eqnarray}\label{eq:exp_T_l-1_l}
\EE_{\ell-1}(T_{\ell})&=& \frac{1}{P_H(\ell,\ell-1)}\sum_{i=1}^{\ell-1}\frac{\pi_H(i)}{\pi_H(\ell)}\, ,
\end{eqnarray}
(see for instance \cite[Section 2.5]{LePeWi09}). 
The following lemma will be useful.

\begin{lemma}\label{lem:var_Tk}
Let $0<\beta<1$ and $K = (1-\beta)\frac{n}{2}$. Then there exist constants $a_\beta, b_\beta\in\R$ depending on $\beta$ only such that
\begin{eqnarray*}
\EE_1(T_K)&\leq & n\log n +a_\beta n\, ,
\end{eqnarray*}
and
\begin{eqnarray*}
\var_1 T_K \leq b_\beta n^2\, .
\end{eqnarray*}
\end{lemma}
\begin{proof}[Proof of Lemma \ref{lem:var_Tk}]
For $2\leq k\leq K$, let $\mu_{k}=\EE_{k-1} T_{k}$ and $v_{k}=\var_{k-1}(T_{k})$. Invoking~\eqref{eq:exp_T_l-1_l}, we have
\begin{eqnarray}\label{eq:bound_mu_k}
\mu_k&=&\sum_{i=1}^{k-1} \frac{\binom{n}{i}}{\binom{n-2}{k-2}}\leq  \frac{\binom{n}{k-1}}{\binom{n-2}{k-2}}\sum_{i=1}^{k-1}\left(\frac{k-1}{n-k+2}\right)^{k-i-1}\leq  \frac{n(n-1)}{(k-1)(n-2k+1)}
\end{eqnarray}
Summing from $2$ to $K$ yields the desired bound on $\EE_1 T_K$. Moving on to the variance, by independence of the successive hitting times, we have
\begin{eqnarray*}
\var_1 T_K&=& \sum_{k=1}^{K-1} v_{k+1}\, .
\end{eqnarray*} 
Hence, it is sufficient to show that there exists a constant $c_\beta>0$ such that $v_{k+1}\leq \frac{c_\beta n^2}{k^2}$ for all $k\leq K$. To do so, we consider the following distributional identity for the hitting time $T_{k+1}$ starting from $k$:
\begin{eqnarray*}
T_{k+1}&=& 1+(1-I)\ti{T}_{k+1} + IJ(\widehat{T}_k +\widehat{T}_{k+1})\, ,
\end{eqnarray*}
where $I$ is the indicator that the chain moves (\textit{i.e.} that a one is picked as updating coordinate), $J$ is the indicator that the chain decreases given that it moves (\textit{i.e.} that the chosen one is added to another one), $\ti{T}_{k+1}$ and $\widehat{T}_{k+1}$ are copies of $T_{k+1}$, and $\widehat{T}_k$ is the hitting time of $k$ starting from $k-1$. All those variables may be assumed to be independent. We obtain the following induction relation:
\begin{eqnarray*}
v_{k+1}&=& \frac{k-1}{n-1}(v_k+v_{k+1})+\left(1-\frac{k}{n}\right)\mu_{k+1}^2 +\frac{k-1}{n-1}\left(1-\frac{k(k-1)}{n(n-1)}\right)(\mu_k+\mu_{k+1})^2\\
&\leq & \frac{k}{n}(v_k+v_{k+1})+\mu_{k+1}^2 +\frac{k}{n}(\mu_k+\mu_{k+1})^2\, .
\end{eqnarray*}
Using the fact that for all $k\leq K$, we have $\mu_k\leq \frac{n}{\beta k}$ (which can be seen by inequality \eqref{eq:bound_mu_k}), and after some simplification, 
\begin{eqnarray*}
v_{k+1}&\leq & \frac{k}{n-k}v_k+\frac{3n^3}{\beta^2k^2(n-k)}\;\leq \; \frac{k}{n-k}v_k+\frac{6n^2}{\beta^2k^2}\, \cdot 
\end{eqnarray*}
By induction and using that $v_2\leq n^2$, we obtain that $v_{k+1}\leq \frac{c_\beta n^2}{k^2}$ for all $k\leq K$.

\end{proof}

The following proposition establishes cutoff for the chain $\{H_t\}$ and will be used in the next section to prove cutoff for the chain $\{\Z_t\}$.

\begin{proposition}\label{prop:cutoff-hamming}
The chain $H_t$ exhibits cutoff at time $\frac{3}{2}n\log n$ with window $n$. 
\end{proposition}
\begin{proof}
For the lower bound, we want to show that for $t=\frac{3}{2}n\log n -2\alpha n$
\begin{eqnarray*}
d_H(t)&\geq & 1-\varepsilon(\alpha)\, ,
\end{eqnarray*}
where $\varepsilon(\alpha)\to 0$ as $\alpha\to +\infty$.
Consider the chain started at $H_0=1$ and let $k=\frac{n}{2}-\alpha \sqrt{n}$ and $A=\{k,k+1,\dots,n\}$. By definition of total-variation distance,
\begin{eqnarray*}
d_H(t)&\geq & \pi_H(A)-P_H^t(1,A)\,\geq \, \pi_H(A) -\PP_1(T_k\leq t)\, .
\end{eqnarray*}
By the Central Limit Theorem, $\displaystyle{\lim_{\alpha\to\infty}\lim_{n\to \infty}\pi_H(A)=1}$. Moving on to ${\PP_1(T_k\leq t)}$, let us write 
\begin{eqnarray*}
\PP_1(T_k\leq t)&=&\PP_1\left(T_{n/3}\leq n\log n-\alpha n\right)+\PP_{n/3}\left(T_k\leq \frac{n\log n}{2}-\alpha n\right)\, .
\end{eqnarray*}
Note that $T_{n/3}$ is stochastically larger than $\sum_{i=1}^{n/3} G_i$, where $(G_i)_{i=1}^{n/3}$ are independent Geometric random variables with respective parameter $i/n$ (this is because at each step, we need at least to pick a one to just move from the current position). By Chebyshev's Inequality,
\begin{eqnarray*}
\PP_1\left(T_{n/3}\leq n\log n-\alpha n\right) & = & O\left(\frac{1}{\alpha^2}\right)\, . 
\end{eqnarray*}
Now, starting from Hamming weight $n/3$ and up to time $T_k$, we may couple $H_t$ with $\ti{H}_t$, the Hamming weight of the standard lazy random walk on the hypercube (at each step, pick a coordinate uniformly at random and randomize the bit at this coordinate), in such a way $T_k\geq S_k$, where $S_k=\inf\{t\geq 0, \ti{H}_t=k\}$. It is known that $S_k$ satisfies
\begin{eqnarray*}
\PP_{n/3}\left(S_k\leq \frac{n\log n}{2}-\alpha n\right)&\leq& \varepsilon(\alpha)\, ,
\end{eqnarray*}
with $\varepsilon(\alpha)\to 0$ as $\alpha\to+\infty$ (see for instance the proof of \cite[Proposition 7.13]{LePeWi09}), which concludes the proof of the lower bound.

For the upper bound, letting $t=\frac{3}{2}n\log n +2\alpha n$, we have
\begin{eqnarray}\label{eq:d_H}
d_H(t)&\leq & \PP_1\left(T_{n/3}> n\log n+\alpha n\right)+\max_{k\geq n/3}d_H^{(k)}\left(\frac{n\log n}{2}+\alpha n\right)\, .
\end{eqnarray}
Lemma \ref{lem:var_Tk} entails that $T_{n/3}$ concentrates well: $\EE_1 (T_{n/3})=n\log n +cn$ for some absolute constant $c$, and $\var_1(T_{n/3})=O(n^2)$. By Chebyshev's Inequality, 
\begin{eqnarray}\label{eq:T_n/3}
\PP_1\left(T_{n/3}>n\log n +\alpha n\right)&= & O\left(\frac{1}{\alpha^2}\right)\, .
\end{eqnarray}

To control the second term in the right-hand side of \eqref{eq:d_H}, we use the coupling method (see \citet[Corollary 5.3]{LePeWi09}). For all starting point $k\geq n/3$, we consider the following coupling between a chain $H_t$ started at $k$ and a chain $H_t^\pi$ started from stationarity: at each step $t$, if $H_t$ makes an actual move (a one is picked as updating bit in the underlying chain $\Z_t$), we try ``as much as possible'' not to move $H_t^\pi$ (picking a zero as updating bit). Conversely, when $H_t$ does not move, we try ``as much as possible'' to move $H_t^\pi$, the goal being to increase the chance that the two chains do not cross each other (by moving at the same time). The chains stay together once they have met for the first time. We claim that the study of the coupling time can be reduced to the study of the first time when the chain started at $n/3$ reaches $n/2$. Indeed, when both chains have reached $n/2$, either they have met, or they have crossed each other. In this last situation, we know however that the expected time of their first return to $n/2$ is $O(\sqrt{n})$, so that $\PP_{n/2}\left(T_{n/2}^+>\sqrt{\alpha n}\right)=O(1/\sqrt{\alpha})$. Moreover, thanks to our coupling, during each of those excursions, the chains have positive probability to meet, so that after an additional time of order $\alpha\sqrt{n}$ we can guarantee that they have met with large probability.  Moreover, as $\pi_H([2n/3,n])=o(1)$, with high probability, $H_0^\pi\leq 2n/3$, and as starting from a larger Hamming weight can only speed up the chain, $\PP_{2n/3}(T_{n/2}>t)\leq \PP_{n/3}(T_{n/2}>t)$. We are thus left to prove that $\PP_{n/3}\left(T_{n/2}>\frac{n\log n}{2}+\alpha n\right)\leq \varepsilon(\alpha)$, for a function $\varepsilon$ tending to $0$ at $+\infty$.

Starting from $H_0=n/3$, we first argue that $H_t$ will remain above $2n/7$ for a very long time. Namely, defining $\displaystyle{\cG_t=\left\{T_{2n/7}>t\right\}}$, we have 
\begin{eqnarray}\label{eq:good-hamming}
\PP_{n/3}\left(\cG_{n^2}\right)&=&1-o(1)\, .
\end{eqnarray}
This can easily be seen by considering $T_k^+=\min\{t\geq 1,\, H_t=k\}$ and taking a union bound over the excursions around $k=n/3$ which visit $m=2n/7$:
\begin{eqnarray*}
\PP_{k}(T_m\leq n^2) &\leq & n^2\PP_k(T_m\leq T_k^+)\, ,
\end{eqnarray*}
and
\begin{eqnarray*}
\PP_k(T_m\leq T_k^+)&=&\frac{\EE_k(T_k^+)}{\EE_m(T_k)+\EE_k(T_m)}\,\leq\, \frac{\EE_k(T_k^+)}{\EE_m(T_m^+)} \,=\, \frac{\pi_H(m)}{\pi_H(k)} \, ,
\end{eqnarray*}
which decreases exponentially fast in $n$.

Our goal now will be to analyse the tail of $\tau=\inf\{t\geq 0,\, D_t\leq 0\}$, where
\begin{eqnarray*}
D_t&=& \frac{n}{2}-H_{t}\, .
\end{eqnarray*}
Observe that
\begin{eqnarray}\label{eq:D_t-moves}
D_{t+1}-D_t= 
\begin{cases}
1 &\mbox{with probability $\frac{H_t(H_t-1)}{n(n-1)}$\, }\vspace{2mm}\\

-1 &\mbox{with probability $\frac{H_t(n-H_t)}{n(n-1)}$\, }
\vspace{2mm}\\
0 &\mbox{otherwise.}
\end{cases}
\end{eqnarray}
We get
\begin{eqnarray}\label{eq:D_t_cond}
\EE\left[D_{t+1}-D_t\given D_t\right]&=&-\frac{2\left(\frac{n}{2}-D_t\right)\left(D_t+1\right)}{n(n-1)}\; \leq \; -\frac{D_t}{n} +\frac{2D_t^2}{n(n-1)}\, \cdot
\end{eqnarray}
Writing a similar recursion for the second moment of $D_t$ gives
\begin{eqnarray*}
\EE\left[D_{t+1}^2-D_t^2\given D_t\right]
&=& -\frac{4H_tD_t(D_t+1/2)}{n(n-1)} +\frac{H_t}{n}\;\leq \; -\frac{4H_tD_t^2}{n^2} +2\, .
\end{eqnarray*}
On the event $\cG_t$, 
\begin{eqnarray*}
\EE\left[D_{t+1}^2-D_t^2\given D_t\right]
&\leq & -\frac{8D_t^2}{7n} +2\, .
\end{eqnarray*}
By induction, letting $\cD_t=\ind_{\cG_t}D_t$ (and noticing that $\cG_{t+1}\subset\cG_t$), we get
\begin{eqnarray*}
\EE\left[\cD_t^2\right]&\leq & \EE[D_0^2]\left(1-\frac{8}{7n}\right)^t + \frac{7n}{4}\;\leq \; \frac{n^2}{4}\e^{-8t/7n} +2n\, .
\end{eqnarray*}
Plugging this back in \eqref{eq:D_t_cond},
\begin{eqnarray*}
\EE\left[\cD_{t+1}\right]&\leq & \left(1-\frac{1}{n}\right) \EE\left[\cD_t\right]+\e^{-8t/7n} +4/n\, ,
\end{eqnarray*}
and by induction,
\begin{eqnarray}\label{eq:exp_cd_t}
\EE\left[\cD_t\right]&\leq & an\e^{-t/n}+b\, ,
\end{eqnarray}
for absolute constants $a,b\geq 0$. Also, letting $\tau_\star=\inf\{t\geq 0,\, \cD_t=0\}$, we see by \eqref{eq:D_t-moves} that, provided $\tau_\star > t$, the process $\{\cD_t\}$ is at least as likely to move downwards than to move upwards and that there exists a constant $\sigma^2>0$ such that $\var\left(\cD_{t+1}\given \cD_t\right)\geq \sigma^2$ (this is because, on $\cG_t$ the probability to make a move at time $t$ in larger than some positive absolute constant). By \citet[Proposition 17.20]{LePeWi09}, we know that for all $u>0$ and $k\geq 0$,
\begin{eqnarray}\label{eq:random-walk}
\PP_{k}(\tau_\star >u)&\leq & \frac{4k}{\sigma\sqrt{u}}\, .
\end{eqnarray}
Now take $H_0=n/3$, $D_0=n/6$, $s=\frac{1}{2}n\log n$ and $u=\alpha n$. We have
\begin{eqnarray*}
\PP_{D_0}\left(\tau >s+u\right)&\leq & \PP_{\cD_0}\left(\tau_\star >s+u\right)+\PP_{H_0}\left({\cG^c_{n^2}}\right)\, .
\end{eqnarray*}
By equation \eqref{eq:good-hamming}, $\PP_{H_0}\left({\cG^c_{n^2}}\right)=o(1)$, and, combining \eqref{eq:random-walk} and \eqref{eq:exp_cd_t}, we have
\begin{eqnarray*}
\PP_{\cD_0}\left(\tau_\star >s+u\right)
&= & \EE_{\cD_0}\left[
\PP_{\cD_s}\left(\tau_\star>u\right)\right]\; \leq \; \EE_{\cD_0}\left[\frac{4\cD_s}{\sigma\sqrt{u}}\right]\; =\; O\left(\frac{1}{\sqrt{\alpha}}\right)\, ,
\end{eqnarray*}
which implies
\begin{eqnarray}\label{eq:d_H_n/3}
\max_{k\geq n/3}d_H^{(k)}\left(s+u\right)&=& O\left(\frac{1}{\sqrt{\alpha}}\right)\, ,
\end{eqnarray}
and concludes the proof of the upper bound.
\end{proof}

\section{Proof of Theorem \ref{thm:cutoff_p=2}}
\label{sec:proof}

First note that, as projections of chains can not increase total-variation distance, the lower bound on $d(t)$ readily follows from the lower bound on $d_H(t)$, as established in Proposition \ref{prop:cutoff-hamming}. Therefore, we only have to prove the upper bound.

Let $\cE=\{x\in\cX,\, H(x)\geq n/3\}$ and $\tau_\cE$ be the hitting time of set $\cE$. For all $t,s>0$, we have
\begin{eqnarray*}
d(t+s)&\leq & \max_{x_0\in\cX}\PP_{x_0}\left(\tau_{\cE}>s\right)+\max_{x\in\cE}d_{x}(t)\, .
\end{eqnarray*}
By \eqref{eq:T_n/3}, taking $s=n\log n +\alpha n$, we have $\max_{x_0\in\cX}\PP_{x_0}(\tau_\cE>s)=O(1/\alpha^2)$, so that our task comes down to showing that for all $x\in\cE$,
\begin{eqnarray*}
d_{x}\left(\frac{n\log n}{2} +\alpha n\right)\leq \varepsilon(\alpha)\, ,
\end{eqnarray*}
with $\varepsilon(\alpha)\to 0$ as $\alpha\to +\infty$. Let us fix $x\in\cE$. Without loss of generality, we may assume that $x$ is the vertex with $\hx\geq n/3$ ones on the first $\hx$ coordinates, and $n-\hx$ zeros on the last $n-\hx$ coordinates. We denote by $\{\Z_t\}$ the random walk started at $\Z_0=x$ and for a vertex $z\in\cX$, we define a two-dimensional object $\W(z)$, keeping track of the number of ones within the first $\hx$ and last $n-\hx$ coordinates of $z$, that is
\begin{eqnarray*}
\W(z)&=& \left(\sum_{i=1}^{\hx} z(i),\, \sum_{i=\hx+1}^{n} z(i)\right)\, .
\end{eqnarray*}
The projection of $\{\Z_t\}_{t\geq 0}$ induced by $\W$ will be denoted $\W_t=\W(\Z_t)=(\X_t,\Y_t)$. We argue that the study of $\{\Z_t\}_{t\geq 0}$ can be reduced to the study of $\{\W_t\}_{t\geq 0}$, and that, when coupling two chains distributed as $\W_t$, we can restrict ourselves to initial states with the same total Hamming weight. Indeed, letting $\nu_{\hx}$ be the uniform distribution over $\{z\in\cX,\, H(z)=\hx\}$, by the triangle inequality 
\begin{eqnarray}\label{eq:triangle}
d_{x_0}(t)
&\leq & \big\| \PP_{x}\left(\Z_t\in\cdot\right)-\PP_{\nu_{\hx}}\left(\Z_t\in\cdot\right)\big\|_{\textsc{tv}} + 
\big\| \PP_{\nu_{\hx}}\left(\Z_t\in\cdot\right)-\pi(\cdot)\big\|_{\textsc{tv}}
\end{eqnarray}
Starting from $\nu_{\hx}$, the conditional distribution of $\Z_t$ given $\{H(\Z_t)=h\}$ is uniform over $\{y\in\cX,\, H(y)=h\}$. This entails
\begin{eqnarray*}
\big\| \PP_{\nu_{\hx}}\left(\Z_t\in\cdot\right)-\pi(\cdot)\big\|_{\textsc{tv}}&=& \big\| \PP_{\hx}\left(H_t\in\cdot\right)-\pi_H(\cdot)\big\|_{\textsc{tv}}\, .
\end{eqnarray*}
For $t=\frac{n\log n}{2} +\alpha n$, we know by \eqref{eq:d_H_n/3} in the proof of Proposition \ref{prop:cutoff-hamming} that 
\begin{eqnarray*}
\big\| \PP_{\hx}\left(H_t\in\cdot\right)-\pi_H(\cdot)\big\|_{\textsc{tv}}&=&O\left(1/\sqrt{\alpha}\right)\, . 
\end{eqnarray*}
As for the first term in the right-hand side of \eqref{eq:triangle}, note that if $z$ and $z'$ are two vertices such that $\W(z)=\W(z')$, then for all $t\geq 0$, $\PP_{x}(\Z_t=z)= \PP_{x}(\Z_t=z')$, and that for all $y\in\cX$ such that $\W(y)=(k,\ell)$
\begin{eqnarray*}
\PP_{\nu_{\hx}}\left(\Z_t=y\right)&=& \sum_{\substack{i,j \\ i+j=\hx}}\sum_{z,\,\W(z)=(i,j)}\frac{1}{\binom{n}{\hx}}\PP_z\left(\Z_t=y\right)\\
&=& \sum_{\substack{i,j \\ i+j=\hx}}\frac{\binom{\hx}{i}\binom{n-\hx}{j}}{\binom{n}{\hx}}\sum_{z,\,\W(z)=(i,j)}\frac{\PP_z\left(\Z_t=y\right)}{\binom{\hx}{i}\binom{n-\hx}{j}}\\
&=& \sum_{\substack{i,j \\ i+j=\hx}}\frac{\binom{\hx}{i}\binom{n-\hx}{j}}{\binom{n}{\hx}}\frac{\PP_{(i,j)}\left(\W_t=(k,\ell)\right)}{\binom{\hx}{k}\binom{n-\hx}{\ell}}\, \cdot
\end{eqnarray*}
Hence,
\begin{eqnarray*}
\big\| \PP_{x}\left(\Z_t\in\cdot\right)-\PP_{\nu_{\hx}}\left(\Z_t\in\cdot\right)\big\|_{\textsc{tv}}
&\leq & \max_{\substack{i,j \\ i+j=\hx}}\big\| \PP_{(\hx,0)}\left(\W_t\in\cdot\right)-\PP_{(i,j)}\left(\W_t\in\cdot\right)\big\|_{\textsc{tv}}\, .
\end{eqnarray*}

Now let $y\in\cE$ such that $H(y)=\hx$, and consider the chains $\Z_t, \ti{\Z}_t$ started at $x$ and $y$ respectively. Let $\W(\Z_t)=(\X_t,\Y_t)$ and $\W(\ti{\Z}_t)=(\ti{\X}_t,\ti{\Y}_t)$. We couple $Z_t$ and $\ti{\Z}_t$ as follows: at each step $t$, provided $H(\Z_t)=H(\ti{\Z}_t)$ and $\W(\Z_t)\neq\W(\ti{\Z}_t)$, we consider a random permutation $\pi_t$ which is such that $\Z_t(i)=\ti{\Z}_t(\pi_t(i))$ for all $1\leq i\leq n$, that is, we pair uniformly at random the ones (resp. the zeros) of $\Z_t$ with the ones (resp. the zeros) of $\ti{\Z}_t$ (one such pairing of coordinates is depicted in Figure \ref{fig:coupling}). If $\Z_t$ moves to $\Z_{t+1}$ by choosing the pair $(i_t,j_t)$ and updating $\Z_{t}(j_t)$ to $\Z_{t}(j_t)+\Z_{t}(i_t)$, then we move from $\ti{\Z}_t$ to $\ti{\Z}_{t+1}$ by updating $\ti{\Z}_t(\pi_t(j_t))$ to $\ti{\Z}_t(\pi_t(j_t))+\ti{\Z}_t(\pi_t(i_t))$. Once $\W(\Z_t)=\W(\ti{\Z}_t)$, the permutation $\pi_t$ is chosen in such a way that the ones in the top (resp. in the bottom) in $Z_t$ are matched with the ones in the top (resp. in the bottom) in $\ti{\Z}_t$, guaranteeing that from that time $\W(\Z_t)$ and $\W(\ti{\Z}_t)$ remain equal. Note that this coupling ensures that for all $t\geq 0$, the Hamming weight of $\Z_t$ is equal to that of $\ti{\Z}_t$, and we may unequivocally denote it by $H_t$. In particular, coupling of the chains $\W(\Z_t)$ and $\W(\ti{\Z}_t)$ occurs when $\X_t$ and $\ti{\X}_t$ are matched. As $\X_t\geq\ti{\X}_t$ for all $t\geq 0$, we may consider
\begin{eqnarray*}
\tau=\inf\{t\geq 0,\, \bD_t=0\}\, ,
\end{eqnarray*}
where $\bD_t=\X_t-\ti{\X}_t$.

\begin{figure}
\centering
\begin{tikzpicture}[scale=0.6,>=latex]

\node[blue] at (2,13) {$1$};
\node at (2,12) {$0$};
\node[blue] at (2,11) {$1$};
\node[blue] at (2,10) {$1$};
\node[blue] at (2,9) {$1$};
\node at (2,8) {$0$};
\node[blue] at (2,7) {$1$};
\node at (2,6) {$0$};
\node at (2,5) {$0$};
\node at (2,4) {$0$};
\node[blue] at (2,3) {$1$};
\node[blue] at (2,2) {$1$};
\node at (2,1) {$0$};

\node at (8,13) {$0$};
\node[blue] at (8,12) {$1$};
\node at (8,11) {$0$};
\node at (8,10) {$0$};
\node[blue] at (8,9) {$1$};
\node[blue] at (8,8) {$1$};
\node at (8,7) {$0$};
\node[blue] at (8,6) {$1$};
\node at (8,5) {$0$};
\node[blue] at (8,4) {$1$};
\node[blue] at (8,3) {$1$};
\node at (8,2) {$0$};
\node[blue] at (8,1) {$1$};

\draw[dashed,color=red] (-2,6.5) -- (12,6.5);

\draw[color=blue!50!green] (2.2,13) -- (7.8,9);
\draw[dotted,color=brown!50!black,line width=0.4mm] (2.2,12) -- (7.8,7);
\draw[color=blue!50!green](2.2,11) -- (7.8,12);
\draw[color=blue!50!green] (2.2,10) -- (7.8,4);
\draw[color=blue!50!green] (2.2,9) -- (7.8,6);
\draw[dotted,color=brown!50!black,line width=0.4mm] (2.2,8) -- (7.8,13);
\draw[color=blue!50!green] (2.2,7) -- (7.8,3);

\draw[dotted,color=brown!50!black,line width=0.4mm] (2.2,6) -- (7.8,10);
\draw[dotted,color=brown!50!black,line width=0.4mm] (2.2,5) -- (7.8,2);
\draw[dotted,color=brown!50!black,line width=0.4mm] (2.2,4) -- (7.8,11);
\draw[color=blue!50!green] (2.2,3) -- (7.8,8);
\draw[color=blue!50!green] (2.2,2) -- (7.8,1);
\draw[dotted,color=brown!50!black,line width=0.4mm] (2.2,1) -- (7.8,5);

\draw[<->] (1,6.6) -- (1,13);
\node at (0.5,10) {$\hx$};
\draw[<->] (1,1) -- (1,6.4);
\node at (0,3.5) {$n-\hx$};
\end{tikzpicture}
\caption{A pairing of coordinates of $\Z_t$ and $\ti{\Z}_t$ .}
\label{fig:coupling}
\end{figure}
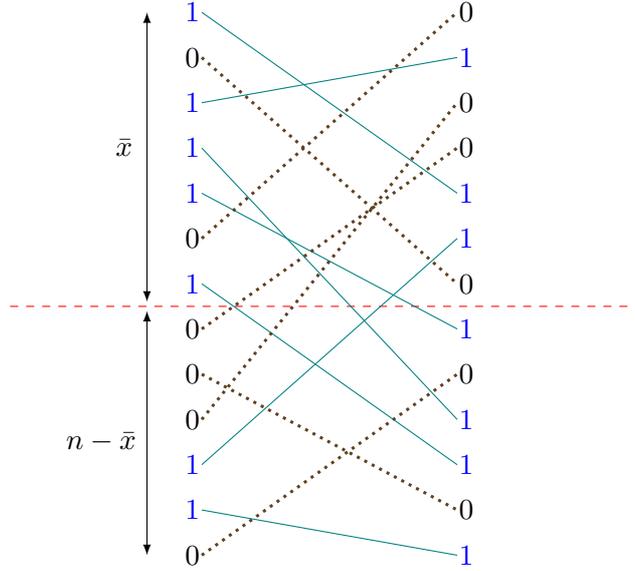

Before analyzing the behavior of $\{\bD_t\}$, we first notice that the worst possible $y$ for the coupling time satisfies $\W(y)=\left(\max\{0,2\hx-n\},\min\{\hx,n-\hx\}\right)$. We now fix $y$ to be such a vertex, and show that, starting from $x,y$, the variables $\W(\Z_t), \W(\ti{\Z}_t)$ remain ``nice'' for a very long time. More precisely, defining
\begin{eqnarray*}
\cB_t&=&\bigcap_{s=0}^t\left\{H_s\geq 2n/7,\, \X_s\geq \frac{\hx}{p},\, \ti{\Y}_s\geq \frac{\min\{\hx,n-\hx\}}{p}\right\}\, ,
\end{eqnarray*}
we claim that we can choose $p \geq 1$ fixed such that
\begin{eqnarray}\label{eq:good-hamming-2}
\PP_{x,y}\left(\cB_{n^2}\right)&=&1-o(1)\, .
\end{eqnarray}
Indeed, the fact that $\PP_{n/3}(T_{2n/7}\leq n^2)=o(1)$ has already been established in the proof of Proposition \ref{prop:cutoff-hamming} (equation \eqref{eq:good-hamming}). Let us show, with the same kind of arguments, that $\displaystyle{\PP_{(\hx,0)}\left(\cup_{s=0}^{n^2} \{\X_s<\hx/p\}\right)=o(1)}$. Letting $A=\{(\hx/p,\ell),\, \ell=0,\dots,n-\hx\}$, $\pi_\W$ be the stationary distribution of $\W_t$, and $k_{\hx}=\min\left\{\frac{\hx}{2},\frac{n-\hx}{2}\right\}$, we have
\begin{eqnarray*}
\PP_{(\hx,0)}(T_A\leq n^{2})&\leq & \PP_{(\hx/2,k_{\hx})}(T_A\leq n^{2})\;\leq \; n^{2}\sum_{\ell=0}^{n-\hx} \PP_{(\hx/2,k_{\hx})}\left(T_{(\hx/p,\ell)}\leq T_{(\hx/2,k_{\hx})}^+\right)\\
&\leq & n^{2}\sum_{\ell=0}^{n-\hx}\frac{\pi_\W(\hx/p,\ell)}{\pi_\W(\hx/2,k_{\hx})}\;=\; \frac{n^{2} 2^{n-\hx}\binom{\hx}{\hx/p}}{\binom{\hx}{\hx/2}\binom{n-\hx}{k_{\hx}}}\, ,
\end{eqnarray*}
and we can choose $p$ large enough such that this quantity decreases exponentially fast in $n$. Similarly, starting from $y$, the value of $\ti{Y}_s$ will remain at a high level for a very long time, establishing \eqref{eq:good-hamming-2}.

Let us now turn to the analysis of $\{\bD_t\}$. On the event $\{t<\tau\}$,
\begin{eqnarray}\label{eq:bD_t-moves}
\bD_{t+1}-\bD_t= 
\begin{cases}
1 &\mbox{with probability $p_1^t$\, }\vspace{2mm}\\

-1 &\mbox{with probability $p_{-1}^t$\, }
\vspace{2mm}\\
0 &\mbox{otherwise,}
\end{cases}
\end{eqnarray}
where 
\begin{eqnarray*}
p_1^t&=& \frac{H_t}{n}\cdot\frac{n-H_t}{n-1}\cdot\frac{\hx-\X_t}{n-H_t}\cdot \frac{n-\hx-\ti{\Y}_t}{n-H_t} + \frac{H_t}{n}\cdot\frac{H_t-1}{n-1}\cdot\frac{\Y_t}{H_t}\cdot \frac{\ti{\X}_t}{H_t}\, ,
\end{eqnarray*}
and
\begin{eqnarray*}
p_{-1}^t&=& \frac{H_t}{n}\cdot\frac{n-H_t}{n-1}\cdot\frac{\hx-\ti{\X}_t}{n-H_t}\cdot \frac{n-\hx-\Y_t}{n-H_t} + \frac{H_t}{n}\cdot\frac{H_t-1}{n-1}\cdot\frac{\X_t}{H_t}\cdot \frac{\ti{\Y}_t}{H_t}\, \cdot 
\end{eqnarray*}
After computation, we get, on $\{t<\tau\}$,
\begin{equation}\label{eq:exp_cond_bD_t}
\EE\left[\bD_{t+1}-\bD_t\given Z_t, \ti{Z}_t\right]=-\frac{H_t\bD_t }{n(n-1)}\left(1+\frac{H_t-1}{H_t}\right)\leq  -\frac{\bD_t}{n^2} \left(2H_t-1\right)
\end{equation}
From \eqref{eq:exp_cond_bD_t}, it is not hard to see that the variable
\begin{eqnarray*}
M_t&=& \ind_{\{\tau>t\}}\bD_t\exp\left(\sum_{s=0}^{t-1}\frac{(2H_s-1)}{n^2}\right)
\end{eqnarray*}
is a super-martingale, which implies $\EE_{x,y}\left[M_t\right]\leq  \EE_{x,y}\left[\bD_0\right] \leq n$. 

Now let $\tau_\star=\inf\{t\geq 0,\,\ind_{\cB_t}\bD_t=0\}$. By \eqref{eq:bD_t-moves}, we see that, provided $\{\tau_\star>t\}$, the process $\{\ind_{\cB_t}\bD_t\}$ is a supermartingale ($p_{-1}^t\geq p_1^t$) and that there exists a constant $\sigma^2>0$ such that the conditional variance of its increments is larger than $\sigma^2$ (because on $\cB_t$, the probability to make a move $p_{-1}^t+p_1^t$ is larger than some absolute constant). By \citet[Proposition 17.20]{LePeWi09}, for all $u>0$ and $k\geq 0$, 
\begin{eqnarray}\label{eq:random-walk-2}
\PP_k(\tau_\star>u)&\leq & \frac{4k}{\sigma\sqrt{u}}\, \cdot
\end{eqnarray}

Now take $t=\frac{n\log n}{2}$ and $u=\alpha n$. We have
\begin{eqnarray*}
\PP_{x,y}(\tau>t+u)&\leq & \PP_{x,y}(\cB_{n^2}^c)+\PP_{x,y}(\tau_\star>t+u)\, .
\end{eqnarray*}
By \eqref{eq:good-hamming-2}, we know that $\PP_{x,y}(\cB_{n^2}^c)=o(1)$. Also, considering the event
\begin{eqnarray*}
\cA_{t-1}&=& \left\{\sum_{s=0}^{t-1}H_s\geq \frac{n^2\log n}{4}-\beta n^2\right\}\, ,
\end{eqnarray*}
and invoking~\eqref{eq:random-walk-2}, we get 
\begin{eqnarray*}
\PP_{x,y}\left(\tau_\star > t+u\right)&\leq & \EE_{x,y}\left[\ind_{\{\tau_\star >t\}}\PP_{Z_t,\ti{Z}_t} \left(\tau_\star > u\right)\right]\\
&\leq & \PP_{x,y}\left(\{\tau_\star>t\}\cap \cA_{t-1}^c\right) + \EE_{x,y}\left[\ind_{\cA_{t-1}}\ind_{\{\tau_\star>t\}} \frac{4\bD_t}{\sigma\sqrt{u}}\right]\, \cdot 
\end{eqnarray*}

On the one hand, recalling the notation and results of Section \ref{sec:hamming} (in particular equation \eqref{eq:exp_cd_t}), and applying Markov's Inequality, 
\begin{eqnarray*}
\PP_{x,y}\left(\{\tau_\star>t\}\cap\cA_{t-1}^c\right)&\leq & \PP_{x,y}\left(\sum_{s=0}^{t-1}\cD_s>\beta n^2\right)\\
&\leq &\frac{1}{\beta n^2}\sum_{s=0}^{t-1}\left(an\e^{-s/n}+b\right)\; =\; O\left(\frac{1}{\beta}\right)\, \cdot 
\end{eqnarray*}

On the other hand,
\begin{eqnarray*}
\EE_{x,y}\left[\ind_{\{\tau_\star>t\}}\ind_{\cA_{t-1}}\bD_t\right]
&\leq & \exp\left(-\frac{\log n}{2}+\frac{t}{n^2}+2\beta\right)\EE_{x,y}\left[M_t\right]\; =\; O\left(\e^{2\beta}\sqrt{n}\right)\, .
\end{eqnarray*}

In the end, we get
\begin{eqnarray*}
\PP_{x,y}\left(\tau> t+u\right)&=&O\left(\frac{1}{\beta}+\frac{\e^{2\beta}}{\sqrt{\alpha}}\right)\, \cdot 
\end{eqnarray*}

Taking for instance $\beta=\frac{1}{5}\log\alpha$ concludes the proof of Theorem \ref{thm:cutoff_p=2}.

\subsection*{Acknowledgements.}

We thank Persi Diaconis and Ohad Feldheim for helpful discussions and references. We thank Ryokichi Tanaka and Alex Lin Zhai for valuable comments and corrections, as well as Markus Heydenreich and his student Manuel Sommer for pointing out an issue in the previous version of equation~\eqref{eq:bound_mu_k}. We also thank Ron Rivest for suggesting this problem to us.

\bibliographystyle{abbrvnat}
\bibliography{biblio}

\begin{thebibliography}{7}
\providecommand{\natexlab}[1]{#1}
\providecommand{\url}[1]{\texttt{#1}}
\expandafter\ifx\csname urlstyle\endcsname\relax
  \providecommand{\doi}[1]{doi: #1}\else
  \providecommand{\doi}{doi: \begingroup \urlstyle{rm}\Url}\fi

\bibitem[Andr{\'e}n et~al.(2007)Andr{\'e}n, Hellstr{\"o}m, and
  Markstr{\"o}m]{andren2007complexity}
D.~Andr{\'e}n, L.~Hellstr{\"o}m, and K.~Markstr{\"o}m.
\newblock On the complexity of matrix reduction over finite fields.
\newblock \emph{Advances in applied mathematics}, 39\penalty0 (4):\penalty0
  428--452, 2007.

\bibitem[Christofides(2014)]{christofides2014asymptotic}
D.~Christofides.
\newblock The asymptotic complexity of matrix reduction over finite fields.
\newblock \emph{arXiv preprint arXiv:1406.5826}, 2014.

\bibitem[Chung and Graham(1997)]{chung1997stratified}
F.~R.~K. Chung and R.~L. Graham.
\newblock Stratified random walks on the n-cube.
\newblock \emph{Random Structures and Algorithms}, 11\penalty0 (3):\penalty0
  199--222, 1997.

\bibitem[Diaconis and Saloff-Coste(1996)]{diaconis1996walks}
P.~Diaconis and L.~Saloff-Coste.
\newblock Walks on generating sets of abelian groups.
\newblock \emph{Probability theory and related fields}, 105\penalty0
  (3):\penalty0 393--421, 1996.

\bibitem[Kassabov(2005)]{kassabov2005kazhdan}
M.~Kassabov.
\newblock Kazhdan constants for {SL}$_n(\mathbb{Z})$.
\newblock \emph{International Journal of Algebra and Computation}, 15\penalty0
  (05n06):\penalty0 971--995, 2005.

\bibitem[Levin et~al.(2009)Levin, Peres, and Wilmer]{LePeWi09}
D.~Levin, Y.~Peres, and E.~Wilmer.
\newblock \emph{Markov chains and mixing times}.
\newblock AMS Bookstore, 2009.

\bibitem[Sotiraki(2016)]{sotiraki2016authentication}
A.~Sotiraki.
\newblock \emph{Authentication protocol using trapdoored matrices}.
\newblock PhD thesis, Massachusetts Institute of Technology, 2016.

\end{thebibliography}

\end{document}